\documentclass[10pt, a4paper]{amsart} 
\usepackage[left=2.50cm, right=2.50cm, top=2.50cm, bottom=2.50cm]{geometry} 
\usepackage{amscd, amsfonts, amsmath, amssymb, amsthm, arydshln, fixmath, graphicx, mathrsfs, overpic, tikz, tikz-cd}
\usepackage{hyperref}
\usepackage[all]{xy} 
\usepackage{tikz}
\usepackage{pgfplots}
\usepackage{pgflibraryarrows}
\usepackage{pgflibrarysnakes}
\usepackage{dsfont}
\makeindex

\theoremstyle{definition}
\newtheorem{Unity}{Unity}[section] 
\newtheorem*{Definition*}{Definition} 
\newtheorem{Definition}[Unity]{Definition}

\theoremstyle{plain} 
\newtheorem*{Theorem*}{Theorem}
\newtheorem{Theorem}[Unity]{Theorem}
\newtheorem{Proposition}[Unity]{Proposition}
\newtheorem{Corollary}[Unity]{Corollary}
\newtheorem{Lemma}[Unity]{Lemma}

\theoremstyle{remark} 
\newtheorem*{Remark*}{Remark}
\newtheorem{Remark}[Unity]{Remark}

\numberwithin{Unity}{section}

\newcommand{\cC}{\mathcal{C}}

\newcommand{\cO}{\mathcal{O}}

\newcommand{\Spec}{\mathrm{Spec\,}}

\newcommand{\Vect}{\mathfrak{Vect}}

\begin{document}

\title{The Birational Invariance of Fundamental Group Schemes}
\author{Lingguang Li}
\address{School of Mathematical Sciences,
Key Laboratory of Intelligent Computing and Applications (Tongji University), Ministry of Education, Shanghai 200092, CHINA}
\email{LiLg@tongji.edu.cn}
\author{Hao Wang}
\address{School of Mathematical Sciences,
Key Laboratory of Intelligent Computing and Applications (Tongji University), Ministry of Education, Shanghai 200092, CHINA}
\email{wanghao0410@tongji.edu.cn}

\begin{abstract}
Let $k$ be a field, $f \colon X \to Y$ a birational morphism of integral connected schemes proper over $k$ with $Y$ normal,  $x \in X(k)$ lying over $y \in Y(k)$. For Tannakian categories $\mathcal{C}_X \subset \mathfrak{Vect}(X)$ and $\mathcal{C}_Y \subset \mathfrak{Vect}(Y)$, denote by $\pi(\mathcal{C}_X,x)$ and $\pi(\mathcal{C}_Y,y)$ the corresponding Tannaka group schemes. We establish a unified Tannakian criteria for the natural homomorphism $\pi(\mathcal{C}_X,x)\to \pi(\mathcal{C}_Y,y)$ to be an isomorphism. As applications, for a birational map $X \dashrightarrow Y$ between smooth projective varieties
over a perfect field $k$, we prove that  there exists a natural isomorphism $\pi^{*}(X,x)\cong \pi^{*}(Y,y)$ for any $* \in \{S,N,EN,F,EF,Loc,ELoc,\acute{e}t, E\acute{e}t,uni\}$. In particular, we prove that  the induced homomorphism $\pi^{str}(X,x)\to \pi^{str}(Y,y)$ is an isomorphism for any birational morphism $ X \rightarrow Y$.
\end{abstract}

\maketitle
\tableofcontents

\section{Introduction}

Fundamental group schemes provide an algebro-geometric refinement of the classical fundamental group. For a connected scheme $X$ proper over a field $k$ endowed with a rational point $x \in X(k)$, a neutral Tannakian subcategory of $\Vect(X)$ determines an affine $k$-group scheme through the fibre functor at $x$. In this way, different classes of vector bundles give rise to different fundamental group schemes. This Tannakian formalism is effective in positive characteristic and over non-algebraically closed fields.

Nori introduced the fundamental group scheme $\pi^{N}(X,x)$,
defined as the Tannaka dual of the category of essentially finite bundles, and the unipotent
fundamental group scheme $\pi^{uni}(X,x)$, attached to the category of unipotent vector bundles
\cite{Nor82}. In positive characteristic, Mehta and Subramanian introduced the local fundamental group
scheme $\pi^{Loc}(X,x)$ by means of Frobenius trivial vector bundles \cite{MeSu08}, Dos Santos introduced the stratified fundamental group scheme $\pi^{str}(X,x)$, defined as the Tannaka dual
of the category of stratified bundles (equivalently, $F$-divisible bundles) \cite{dosSantos07}. Amrutiya and
Biswas introduced the F-fundamental group scheme $\pi^{F}(X,x)$ from the Tannakian subcategory
generated by Frobenius finite bundles \cite{AmBi10}. Langer introduced the S-fundamental group
scheme $\pi^{S}(X,x)$, defined by the category of numerically flat vector bundles \cite{Lan11}. In
characteristic zero, Otabe introduced the EN-fundamental group scheme $\pi^{EN}(X,x)$ through the
category of semi-essentially finite bundles \cite{Ota17}. More recently, Adroja and Amrutiya
introduced the extended local fundamental group scheme $\pi^{ELoc}(X,x)$ by using Frobenius
semi-trivial vector bundles \cite{AdAm25}.
The \'etale
fundamental group scheme $\pi^{\acute{e}t}(X,x)$, associated with the category of \'etale trivializable vector
bundles,  has its origin in Grothendieck's 
\'etale fundamental group \cite{Gro60}.

A basic problem is whether these fundamental group schemes are invariant under birational modifications. For \'etale fundamental group, Grothendieck proved the birational invariance for regular schemes proper over an arbitrary field \cite{Gro60}. For the S-fundamental group scheme, Hogadi and Mehta proved birational invariance for smooth projective varieties over an algebraically closed field of positive characteristic \cite{HM11}. For the F-fundamental group
scheme, Amrutiya  proved birational invariance for smooth projective varieties over a perfect
field of positive characteristic \cite{Am20}. Adroja and Amrutiya also obtained birational invariance results for the local, Nori and extended Nori  fundamental group schemes for smooth projective varieties over an algebraically closed field of positive characteristic \cite{AdAm25}. The aim of this paper is to establish a uniform Tannakian criterion for a birational morphism to induce an isomorphism on the corresponding fundamental group schemes, and then apply this criterion to prove birational invariance for all the fundamental group schemes listed above over an arbitrary perfect field.

Throughout the paper, we assume that $k$ is a field, $f: X \rightarrow Y$ be a morphism between connected schemes proper over $k$, $x \in X(k)$ lying over $y \in Y(k)$. Let $\cC_X $ (resp. $\cC_Y $) be the Tannakian category whose objects consist of vector bundles over $X$ (resp. $Y$)  with neutral fibre functor $\omega_x$ (resp. $\omega_y$). Denote by $\pi(\cC_X,x)$ and $\pi(\cC_Y,y)$ the corresponding Tannaka group schemes. The pullback induces a functor $f^*: \cC_Y \rightarrow \cC_X$ and a canonical
homomorphism  $ \pi(\cC_X,x)\to \pi(\cC_Y,y)$ of Tannaka group schemes.

\begin{Theorem}[Corollary~\ref{birmor}]
   Let $k$ be a field, $f \colon X \to Y$ a  birational morphism of integral schemes proper over $k$ with $Y$ normal, $x \in X(k)$ lying over $y \in Y(k)$. Then the following conditions are equivalent:
\begin{enumerate}
    \item $f^*$ induces an isomorphism
    $ \pi(\cC_X,x)\cong \pi(\cC_Y,y)$.
    \item $f^*$ and $f_*$ are quasi-inverse exact tensor functors between $\cC_X$ and $\cC_Y$.
    \item For any $E\in \cC_X$, we have
    $f_*E\in \cC_Y$.
\end{enumerate}  
\end{Theorem}

In particular, when $f: X \rightarrow Y$ is a birational morphism (or birational map) between smooth projective varieties over $k$, using this theorem we can prove the birational invariance of certain fundamental group schemes.

\begin{Theorem}[Theorem~\ref{birinvariance} and Corollary~\ref{cor:stratified}]
Let $k$ be a perfect field, $f \colon X \dashrightarrow Y$
be a birational map between smooth projective varieties over $k$, $x \in X(k)$ lying over $y \in Y(k)$.  
Then there exists a natural isomorphism
$\pi^{*}(X,x)\cong \pi^{*}(Y,y)$ for any $* \in \{S, N, EN, F, EF, Loc, ELoc, \acute{e}t, E\acute{e}t, uni\}$.
Moreover, if $f$ is a birational morphism, then the induced homomorphism
$\pi^{str}(X,x)\rightarrow \pi^{str}(Y,y)$
is also an isomorphism.
\end{Theorem}

The proof proceeds in two steps. In Section~3 we establish the general criterion above for a birational morphism to induce an isomorphism of Tannaka group schemes. In Section~4 we verify its hypotheses for the Tannakian categories defining the S, Nori, extended Nori, F, extended F, local, extended local, \'etale, extended \'etale and unipotent fundamental group schemes. For numerically flat vector bundles, the key input is the  base change result and the argument of Hogadi--Mehta \cite{HM11}. The remaining cases are then deduced from the numerically flat case and from the defining properties of the corresponding Tannakian subcategories and their saturation behaviors.

\section{Preliminaries}

Let $k$ be a field, $K/k$ a field extension, $f:X\rightarrow Y$ a morphism of schemes over $k$, $y\in Y(k)$, $\Vect(X)$ the category of vector bundles on $X$, $E$ a quasi-coherent sheaf on $X$. Consider the Cartesian diagrams
\[
    \begin{aligned}
        \begin{tikzcd}
     X_y\arrow{r}{t} \arrow{d}{g} & X \arrow{d}{f}\\
     \Spec \kappa(y)\arrow{r}{y} & Y
\end{tikzcd}&\qquad &\begin{tikzcd}
            X\times_{\Spec k}\Spec K\arrow[r,"p"]\arrow[d]& X\arrow[d]\\
            \Spec K\arrow[r]&\Spec k .
        \end{tikzcd}
    \end{aligned}
\]
We denote $X_K := X \times_{\Spec k} \Spec K$ and $E \otimes_k K:= p^*E$. We say a sheaf $E$ on $X$ \emph{satisfies base change at $y$} if the canonical map $y^* f_*E\rightarrow g_* t^* E$
is an isomorphism, i.e. $ f_*E|_y\xrightarrow{\cong} H^0(X_y,E|_{X_y})$.

\begin{Definition}
    Let $k$ be a field, $X$ a connected scheme proper over $k$, $x\in X(k)$, $\cC_X$ a Tannakian category over $X$ whose objects are vector bundles on $X$ with fibre functor $|_x:E\mapsto E|_x$, and denote its Tannaka group scheme by $\pi(\cC_X,x)$.
\end{Definition}

\begin{Definition}
    Let $k$ be a field, $X$ a connected scheme proper over $k$, $x\in X(k)$, $\cC_X$ the Tannakian category over $X$. Define the \textit{saturation category} $\overline{\cC}_X$ of $\cC_X$ as the full subcategory of $\Vect(X)$ whose objects are those $E$ for which there exists a filtration
    $0\hookrightarrow E_1 \hookrightarrow \cdots\hookrightarrow E_n=E,$
    such that $E^i=E_{i+1}/E_i\in\cC_X$ for any $i$.
\end{Definition}

\begin{Remark}
 Let $k$ be a field, $f: X \rightarrow Y$ be a morphism between connected schemes proper over $k$. Let $\cC_X \subset \Vect(X)$ and $\cC_Y \subset \Vect(Y)$ be the Tannakian categories. The pullback  induces a functor $f^*: \cC_Y \rightarrow \cC_X$. Then
 \begin{enumerate}
     \item   The pullback induces a functor $f^*: \overline{\cC}_Y \rightarrow \overline{\cC}_X$.
     \item By \cite[Proposition 3.3]{LiTian26}, $\overline{\cC}_X$ is a Tannakian category over $X$.
 \end{enumerate}
\end{Remark}

\begin{Definition}[\cite{stacks}, Definition~29.50.1]
Let $X,Y$ be schemes with finitely many irreducible components respectively. We say a morphism $f\colon X \to Y$
is \emph{birational} if
\begin{enumerate}
    \item $f$ induces a bijection between the set of generic points of irreducible components of $X$ and the set of generic points of the irreducible components of $Y$.
    \item for every generic point $\eta \in X$ of an irreducible component of $X$, the local ring map
    $\mathcal{O}_{Y,f(\eta)} \to \mathcal{O}_{X,\eta}$
    is an isomorphism.
\end{enumerate}
\end{Definition}

\begin{Lemma}[\cite{stacks}, Lemma~37.53.1]\label{steinfact}
Let $k$ be a field, $f \colon X \to Y$ be a universally closed and quasi-separated morphism between schemes over $k$. There exists a factorization
\[
    \begin{tikzcd}
                                   & Y' \arrow[rd, "\pi"] &   \\
X \arrow[ru, "f'"] \arrow[rr, "f"] &                      & Y
\end{tikzcd}
\]  
with the following properties:
\begin{enumerate}
    \item $Y'=\underline{\mathrm{Spec}}_Y\bigl(f_*\cO_X\bigr)$, $\pi \colon Y' \to Y$ is integral, and $Y'$ is the normalization of $Y$ in $X$.
    \item  $f'$ is universally closed, quasi-compact, quasi-separated, and surjective, and $f'_*\cO_X=\cO_{Y'}$.
   
\end{enumerate}
\end{Lemma}

\begin{Lemma}[\cite{stacks}, Lemma~29.54.8]\label{integraliso}
   Let $k$ be a field, $f:X \rightarrow Y$ be an affine integral birational morphism between integral schemes over $k$ with $Y$ normal. Then $f$ is an isomorphism.
\end{Lemma}

\begin{Corollary}\label{structuresheaves}
  Let $k$ be a field, $f \colon X \to Y$ a  birational morphism of integral schemes proper over  $k$ with $Y$ normal. Then we have $f_*\cO_X = \cO_Y$. 
\end{Corollary}
\begin{proof}
 Since any morphism between proper schemes is proper, and proper morphism is universally closed and quasi-separated, by Lemma~ \ref{steinfact}, we have a commutative diagram
\[
    \begin{tikzcd}
                                   & Y' \arrow[rd, "\pi"] &   \\
X \arrow[ru, "f'"] \arrow[rr, "f"] &                      & Y
\end{tikzcd}
\]   
such that $Y'=\underline{\mathrm{Spec}}_Y\bigl(f_*\cO_X\bigr)$, $f'_*\cO_X = \cO_{Y'}$ and $\pi: Y' \to Y$ is finite. Since $f$ is birational, $\pi$ is also birational,  by Lemma~\ref{integraliso}, $\pi$ is an isomorphism. Hence $f_*\cO_X \cong \pi_*f'_*\cO_X \cong \pi_*\cO_{Y'} \cong \cO_Y$.
    
\end{proof}

\section{The birational invariance of Tannaka group schemes}

\textbf{Notations and Conventions.} Let $k$ be a field, $f: X \rightarrow Y$ a morphism between connected schemes proper over $k$, $x \in X(k)$ lying over $y \in Y(k)$. Let $\Vect(X)$ and $\Vect(Y)$ be the categories of vector bundles on $X$ and $Y$ respectively, $\cC_X \subset \Vect(X)$ and $\cC_Y \subset \Vect(Y)$ be the Tannakian categories. Denote by $\pi(\cC_X,x)$ and $\pi(\cC_Y,y)$ the corresponding Tannaka group schemes.
The pullback induces a functor $f^*: \cC_Y \rightarrow \cC_X$ and a canonical
homomorphism  $ \pi(\cC_X,x)\to \pi(\cC_Y,y)$ of fundamental group schemes.

\begin{Lemma}[\cite{LiTian26}~Lemma 3.1]\label{faith}
    Let $k$ be a field,  $f: X \rightarrow Y$ a morphism of connected schemes proper over $k$  with $f_*\cO_X=\cO_Y $. Then the induced functor $f^*: \Vect(Y)\rightarrow \Vect(X)$ is fully faithful. In particular, the restriction functor $f^*:\cC_Y \rightarrow \cC_X$  is fully faithful.
\end{Lemma}

\begin{Theorem}\label{generalmor}
Let $k$ be a field, $f \colon X \to Y$ a morphism of  connected schemes proper over a field $k$ with $f_*\cO_X=\cO_Y $, $x \in X(k)$ lying over $y \in Y(k)$. Then the following conditions are equivalent:
\begin{enumerate}
    \item $f^*$ induces an isomorphism
    $ \pi(\cC_X,x)\cong \pi(\cC_Y,y)$.
    \item $f^*$ and $f_*$ are quasi-inverse exact tensor functors between $\cC_X$ and $\cC_Y$.
    \item For any $E\in \cC_X$, we have
    $f_*E\in \cC_Y$ and the natural morphism $f^*f_*E \rightarrow E$ is an isomorphism.
    \item For any $E\in \cC_X$,  we have $f_*E\in \cC_Y$ and there exists  $s \in Y(k)$ $($or for any $s \in Y(k))$
    \begin{enumerate}
        \item $E|_{X_{s}}$ is a trivial vector bundle.
        \item $E$ satisfies base change at $s$.
    \end{enumerate}
\end{enumerate} 
\end{Theorem}
\begin{proof}

$(1) \Rightarrow (2)$ Since the induced homomorphism $ \pi(\cC_X,x)\to \pi(\cC_Y,y)$ is an isomorphism, $f^*: \cC_Y \to \cC_X$ is an equivalence of Tannakian categories. Hence for any $E \in \cC_X$, there exists  $F \in \cC_Y$ such that $E \cong f^*F$. Since $f_*\cO_X= \cO_Y$, by projection formula, we have
\[
  f_*E \cong f_*f^*F = f_* (f^*F \otimes_{\cO_X} \cO_X) \cong F \otimes_{\cO_Y} f_*\cO_X  \cong F \in \cC_Y.  
\]
Applying $f^*$ we have $ f^*f_*E \cong f^*F \cong E$. On the other hand, for any $F \in \cC_Y$, by $f_*\cO_X = \cO_Y$ and projection formula, we have   
\[
  f_*f^*F = f_* (f^*F \otimes_{\cO_X} \cO_X) \cong F \otimes_{\cO_Y} f_*\cO_X  \cong F.  
\]
Hence $f^*$ and $f_*$ are quasi-inverse functors.  Since $f_*$ is an equivalence, by \cite[Proposition~2.5.14]{Yeku19},  $f_*$ is an exact functor.
For any $E_1, E_2 \in \cC_X$, by projection formula, we have
\[
f_*(E_1 \otimes_{\cO_X} E_2) \cong f_*(f^*f_*E_1 \otimes_{\cO_X} E_2) \cong f_*E_1 \otimes_{\cO_Y} f_*E_2.
\] 
Hence $f_*$ is a tensor functor. Thus $f^*$ and $f_*$ are quasi-inverse exact tensor functors between $\cC_X$ and $\cC_Y$.

$(2) \Rightarrow (3)$ Obviously.
 
$(3) \Rightarrow (4)$ 
For any $E \in \cC_X$ we have $f_*E \in \cC_Y$ and $f^*f_*E \cong E$. For any $s \in Y(k)$, we have 
\[
    (f_*E)|_s\otimes_k \mathcal{O}_X\cong f^*f_*E|_{X_s}\cong  E|_{X_s}.
\] 
Hence $E|_{X_s}$ is a trivial bundle.
Taking global sections, we have $(f_*E)|_s\cong H^0(X_s,E|_{X_s})$, i.e. $E$ satisfies base change at $s\in Y(k)$.

$(4) \Rightarrow (1)$
Let $E \in \cC_X$ and $s \in Y(k)$ such that $f_*E \in \cC_Y$, $E|_{X_{s}}$ is a trivial  bundle and $E$ satisfies base change at $s$, i.e. $(f_*E)|_s\cong H^0(X_s,E|_{X_s})$. Restricting the morphism $f^*f_*E \rightarrow E$  to $X_s$, we have 
  \[
    (f_*E)|_s\otimes_k \mathcal{O}_X \cong f^*f_*E|_{X_{s}} \rightarrow E|_{X_s}\cong H^0(X_s,E|_{X_{s}})\otimes_{k}\cO_X.  
  \] 
Therefore, we obtain $ f^*f_*E|_{X_s}\cong E|_{X_s}$.
Consider the exact sequence
\[
\ker \hookrightarrow f^{*}f_{*}E \to E \to \operatorname{coker}.
\]
Restricting it to $X_s$, we obtain an exact sequence
\[
\ker|_{X_s} \hookrightarrow f^{*}f_{*}E|_{X_s} \to E|_{X_s} \to \operatorname{coker}|_{X_s}.
\]
Then $\ker|_{X_s}=\operatorname{coker}|_{X_s}=0$. Since $\cC_X$ is an abelian category, we have
$\ker,\operatorname{coker}\in \cC_X$.
It follows that $\ker=\operatorname{coker}=0$, i.e. $f^{*}f_{*}E \cong E$.
Hence $f^*$ is essentially surjective. By Lemma~\ref{faith}, $f^*$ is fully faithful. Hence $f^*$ is an equivalence of categories, which induces an isomorphism $ \pi(\cC_X,x)\cong \pi(\cC_Y,y)$.
\end{proof}

\begin{Corollary}\label{birmor}
   Let $k$ be a field, $f \colon X \to Y$ a  birational morphism of integral schemes proper over $k$ with $Y$ normal, $x \in X(k)$ lying over $y \in Y(k)$. Then the following conditions are equivalent:
\begin{enumerate}
    \item $f^*$ induces an isomorphism
    $ \pi(\cC_X,x)\cong \pi(\cC_Y,y)$.
    \item $f^*$ and $f_*$ are quasi-inverse exact tensor functors between $\cC_X$ and $\cC_Y$.
    \item For any $E\in \cC_X$, we have
    $f_*E\in \cC_Y$.
\end{enumerate}  
\end{Corollary}
\begin{proof}
By Lemma~\ref{structuresheaves} we have $f_*\cO_X =\cO_Y$. 

Let $E \in \cC_X$ such that $f_*E \in \cC_Y$. Consider the exact sequence
\[
\ker \hookrightarrow f^{*}f_{*}E \to E \to \operatorname{coker}.
\]
Since $f$ is birational, there exists an open subset $U \subset X$ such that $f^{*}f_{*}E|_U\cong E|_U $. Then  $\ker|_U =\operatorname{coker}|_U = 0$. Since $\cC_X$ is an abelian category,  we have
$\ker,\operatorname{coker}\in \cC_X$.
It follows that $\ker=\operatorname{coker}=0$, hence the natural homomorphism $f^{*}f_{*}E \rightarrow E$ is an isomorphism. 

Now the Corollary follows from Theorem~\ref{generalmor}.
\end{proof}

\begin{Proposition}\label{saturationbir}
Let $k$ be a field, $f \colon X \to Y$ a morphism of smooth schemes proper over a field $k$ with
$f_*\cO_X \cong \cO_Y$, $x \in X(k)$  lying over $y \in Y(k)$. If $f^*$ induces an isomorphism  $\pi(\cC_X, x) \cong \pi(\cC_Y, y)$, then $f^*$ also induces an isomorphism $\pi(\overline{\cC}_X, x) \cong \pi(\overline{\cC}_Y, y)$.
\end{Proposition}
\begin{proof}

 Since $f^*$ induces an isomorphism  $\pi(\cC_X, x) \cong \pi(\cC_Y, y)$,  by Theorem~\ref{generalmor}, for any $F \in \cC_X$ we have $f^*f_* F \cong F$. For any $F \in \cC_X$ and any $j \geq 0$, by projection formula, we have 
\[
R^jf_*F \cong R^jf_*(f^*f_*F)= R^jf_*(f^*f_*F \otimes_{\cO_X}\cO_X) \cong f_*F \otimes_{\cO_Y} R^jf_*\cO_X.
\]
By \cite[Corollary~3.2.10]{CR11}, we have $R^jf_*\cO_X=0$ for any $j>0$, then $R^jf_*F=0$ for any $j >0$ and $F \in \cC_X$.

For any $E \in \overline{\cC}_X$ there exists a filtration
\[
 0\hookrightarrow E_1 \hookrightarrow \cdots\hookrightarrow E_n=E,
\]
such that $E^i=E_{i+1}/E_i\in\cC_X$ for any $i$. 
Consider the short exact sequence
\[
    0 \rightarrow E_1 \rightarrow E_2 \rightarrow E_2/E_1\rightarrow 0,
\]
applying $f_*$, we have:
\[
    0 \rightarrow f_*E_1 \rightarrow f_*E_2 \rightarrow f_*(E_2/E_1)\rightarrow R^1f_*E_1 \rightarrow R^1f_*E_2 \rightarrow R^1f_*(E_2/E_1) \rightarrow \cdots.
\]
 Hence $f_*E_2 \in \overline{\cC}_Y$ and $R^jf_*E_2 =0$ for any $j>0$. 
Consider the short exact sequence
\[
    0 \rightarrow E_2 \rightarrow E_3 \rightarrow E_3/E_2\rightarrow 0,
\]
applying $f_*$, we have:
\[
    0 \rightarrow f_*E_2 \rightarrow f_*E_3 \rightarrow f_*(E_3/E_2)\rightarrow R^1f_*E_2 \rightarrow R^1f_*E_3 \rightarrow R^1f_*(E_3/E_2) \rightarrow \cdots.
\]
 Hence $f_*E_3 \in \overline{\cC}_Y$ and $R^jf_*E_3 =0$ for any $j>0$. Inductively, we have $f_*E_i \in \overline{\cC}_Y$ and $R^jf_*E_i=0$ for any $i$ and $j>0$. In particular, we have $f_*E \in \overline{\cC}_Y$.
By Corollary~\ref{birmor}, $f^*$ induces an isomorphism $\pi(\overline{\cC}_X, x) \cong \pi(\overline{\cC}_Y, y)$.
\end{proof}

\section{The birational invariance of fundamental group schemes}

\begin{Definition}
Let $k$ be a field, $X$ a geometrically reduced connected scheme proper over $k$, $x\in X(k)$, $E$ a vector bundle on $X$ of rank $r$. If $k$ is of positive characteristic, then $F_X:X\rightarrow X$ is the absolute Frobenius. Then $E$ is said to be
\begin{itemize}
\item \textit{numerically flat}, if both $E$ and $E^\vee$ are nef.
\item \textit{Nori semistable}, if for any smooth projective curve $f:C\rightarrow X$, $f^*E$ is semistable of degree 0 over $C$.
\item \textit{finite}, if there exist $f(t)\neq g(t)\in\mathbb{N}[t]$ such that $f(E)\cong g(E)$, where
    $$h(E):=\bigoplus_{i=0}^m\bigoplus_{j=1}^{n_i}E^{\otimes i}\text{ for any }h(t)=\sum\limits_{i=0}^mn_it^i\in\mathbb{N}[t].$$
\item \textit{essentially finite}, if there exists $E_1\hookrightarrow E_2\hookrightarrow F\in\Vect(X)$ such that $E\cong E_2/E_1$, where $E_1,E_2$ are numerically flat and $F$ is finite.
\item \textit{Frobenius finite}, if there exist $f(t)\neq g(t)\in\mathbb{N}[t]$ such that $\tilde{f}(E)\cong \tilde{g}(E)$, where 
    $$\tilde{h}(E):=\bigoplus\limits_{i=1}^{m}((F_X^i)^*E)^{\oplus n_i} \text{ for any }h(t)=\sum\limits_{i=0}^mn_it^i\in\mathbb{N}[t].$$
\item \textit{essentially Frobenius finite}, if there exists $E_1\hookrightarrow E_2\hookrightarrow F\in\Vect(X)$ such that $E\cong E_2/E_1$, where $E_1,E_2$ are numerically flat and $F$ is Frobenius finite.
\item \textit{Frobenius trivial}, if there exists positive integer $n$ such that $F_X^{n*} E\cong \cO_X^{\oplus r}$.
\item \textit{\'etale trivializable}, if there exists a finite \'etale covering $\phi: P\rightarrow X$ such that $\phi^* E\cong \cO_P^{\oplus r}$.
\item \textit{unipotent}, if there is a filtration
$0 \hookrightarrow E_1 \hookrightarrow \cdots \hookrightarrow E_n = E$
such that
$E_{i+1}/E_i \cong \mathcal{O}_X$
for any $i$.
\item \textit{$F$-divisible}, if it is a sequence
$\mathcal{E}=(E_i,\sigma_i)_{i\ge 0}$,
where each $E_i$ is a vector bundle on $X$ and
$\sigma_i:F_X^*E_{i+1}\xrightarrow{\sim} E_i$
is an isomorphism for every $i\ge 0$.
\end{itemize}
We have the following Tannakian categories:
    \begin{itemize}
        \item $\cC^{NF}(X)$: objects consist of numerically flat bundles on $X$.
        \item $\cC^{N}(X)$: objects consist of essentially finite bundles on $X$.
        \item $\cC^{F}(X)$: objects consist of essentially Frobenius finite bundles on $X$.
        \item $\cC^{Loc}(X)$: objects consist of Frobenius trivial bundles on $X$.
        \item $\cC^{\acute{e}t}(X)$: objects consist of \'etale trivializable bundles on $X$.
        \item $\cC^{uni}(X)$: objects consist of unipotent bundles on $X$.
         \item $\cC^{str}(X)$: objects consist of $F$-divisible bundles on $X$.
    \end{itemize}
We have the following Tannaka group schemes:
\begin{itemize}
        \item $\pi^{S}(X,x):=\pi(\cC^{NF}(X),x)$, called the \textit{S-fundamental group scheme}.
        \item $\pi^{N}(X,x):=\pi(\cC^{N}(X),x)$, called the \textit{Nori fundamental group scheme}.
        \item $\pi^{EN}(X,x):=\pi(\overline{\cC^{N}(X)},x)$, called the \textit{extended Nori fundamental group scheme}.
        \item $\pi^{F}(X,x):=\pi(\cC^{F}(X),x)$, called the \textit{F-fundamental group scheme}.
        \item $\pi^{EF}(X,x):=\pi(\overline{\cC^{F}(X)},x)$, called the \textit{extended F-fundamental group scheme}.
        \item $\pi^{Loc}(X,x):=\pi(\cC^{Loc}(X),x)$, called the \textit{local fundamental group scheme}.
        \item $\pi^{ELoc}(X,x):=\pi(\overline{\cC^{Loc}(X)},x)$, called the \textit{extended local fundamental group scheme}.
        \item $\pi^{\acute{e}t}(X,x):=\pi(\cC^{\acute{e}t}(X),x)$, called the \textit{\'etale fundamental group scheme}.
        \item $\pi^{E\acute{e}t}(X,x):=\pi(\overline{\cC^{\acute{e}t}(X)},x)$, called the \textit{extended \'etale fundamental group scheme}.
        \item $\pi^{uni}(X,x):=\pi(\cC^{uni}(X),x)$, called the \textit{unipotent fundamental group scheme}.
        \item $\pi^{str}(X,x):=\pi(\cC^{str}(X),x)$, called the \textit{stratified fundamental group scheme}.
    \end{itemize}
\end{Definition}

\begin{Remark}
Let $k$ be a field, $X$ a geometrically reduced connected scheme proper over $k$. Then
\begin{enumerate}
\item  For $E \in \Vect(X)$, $E$ is Nori semistable if and only if $E$ is numerically flat.

\item $\cC^{NF}(X)$ and $\cC^{uni}(X)$ are saturated categories, i.e.
$\overline{\cC^{NF}(X)}=\cC^{NF}(X)$
and $\overline{\cC^{uni}(X)}=\cC^{uni}(X)$.

 \item $\cC^{*}(X)$ and $\overline{\cC^{*}(X)}$ are Tannakian subcategories of $\cC^{NF}(X)$ for
    $*\in\{N, F, Loc,\acute{e}t, uni \}$.
\end{enumerate}
\end{Remark}

\begin{Lemma}[\cite{CLM22}, Proposition~1.4.4]\label{nef}
    Let $k$ be a field, $X$ a proper algebraic space over $k$, $E \in \Vect(X) $. Then $E$ is nef if and only if  the pullback $E\otimes_k k'$ on $X_{k'}$ is nef for every field extension $k \subset k'$.
\end{Lemma} 
By Lemma~\ref{nef} we have the following result.
\begin{Lemma}\label{basechange}
   Let $k$ be a field, $K/k$ a field extension and $X$ a connected scheme proper over $k$. Then $E \in \cC^{NF}(X)$ if and only if $E \otimes_k K \in \cC^{NF}(X_K) $. 
\end{Lemma}

\begin{Proposition}\label{birinv:underS}
Let $k$ be a field, $f \colon X \to Y$ a morphism of  connected schemes proper over a field $k$ with $f_*\cO_X=\cO_Y $, $x \in X(k)$ lying over $y \in Y(k)$,  $f^*$ induces the natural isomorphism $\pi^{S}(X,x)\cong \pi^{S}(Y,y)$. Then $f^*$ induces an isomorphism $\pi^{*}(X,x)\cong \pi^{*}(Y,y)$ for  any $* \in \{N, EN, F, EF, Loc, ELoc, \acute{e}t, E\acute{e}t, uni\}$.
\end{Proposition} 
\begin{proof}
By Theorem~\ref{generalmor}, $f_*:\cC^{NF}(X) \rightarrow \cC^{NF}(Y)$ is an exact tensor functor, and for any $E \in \cC^{NF}(X)$, $f_* E \in \cC^{NF}(Y)$ and the natural morphism $f^*f_*E \rightarrow E$ is an isomorphism. 

Let $* \in \{N, EN, F, EF, Loc, ELoc, \acute{e}t, E\acute{e}t, uni\},$ we have $\cC^{*}(X) \subseteq \cC^{NF}(X)$ and $\cC^{*}(Y) \subseteq \cC^{NF}(Y)$, and the natural morphism $f^*f_*E \rightarrow E$ is an isomorphism for any $E \in \cC^*(X)$. In order to prove that the natural homomorphism $\pi^{*}(X,x)\rightarrow \pi^{*}(Y,y)$ is an isomorphism, we only need to show that $f_*E \in \cC^*(Y)$ for any $E \in \cC^*(X)$ by Theorem~\ref{generalmor}.

If $\operatorname{char} (k) >0$, $F_X:X\rightarrow X$ is the absolute Frobenius morphism. Then for any $E \in \Vect(X)$, we have 
\[
f_*F^*_X E  \cong f_*F^*_X f^*f_* E \cong f_*f^*F^*_Y f_*E = f_*(f^* F^*_Yf_* E \otimes_{\cO_X}\cO_X) \cong F^*_Y f_*E \otimes_{\cO_Y} f_*\cO_X = F^*_Y f_*E.
\]
Hence pushforward commutes with Frobenius pullback.

(1) For the case $*= N$.

If $F$ is a finite bundle on $X$, then there exist distinct $g(t),h(t)\in\mathbb{N}[t]$ such that $g(F)\cong h(F)$. Since $f_*$ is  a  tensor functor, we have 
    $g(f_*F)\cong f_*g(F)\cong f_*h(F)\cong h(f_*F)$, 
i.e. $f_*F$ is a finite bundle on $Y$.

If $E\in \cC^N(X)$, then $E\cong E_1/E_2$, where $E_2\hookrightarrow E_1\hookrightarrow F$, $E_1,E_2\in \cC^{NF}(X)$, $F$ is a finite bundle on $X$. Then we have $f_*E_2\hookrightarrow f_*E_1\hookrightarrow f_*F$, where $f_*E_1,f_*E_2\in \cC^{NF}(Y)$ and $f_*F$ is a finite bundle on $Y$. Hence $f_*E\cong f_*E_1/f_*E_2\in \cC^N(Y)$.

(2) For the case $*= F$.

If $F$ is a Frobenius finite bundle on $X$, then there exist distinct $g(t),h(t)\in\mathbb{N}[t]$ such that $\tilde{g}(F)\cong \tilde{h}(F)$. 
It follows that $\tilde{g}(f_*F)\cong f_*\tilde{g}(F)\cong f_*\tilde{h}(F)\cong \tilde{h}(f_*F),$
i.e. $f_*F$ is a Frobenius finite bundle on $Y$.

If $E\in\cC^F(X)$, then $E\cong E_1/E_2$, where $E_2\hookrightarrow E_1\hookrightarrow F$, $E_1,E_2\in \cC^{NF}(X)$ and $F$ is a Frobenius finite bundle on $X$. Then we have $f_*E_2\hookrightarrow f_*E_1\hookrightarrow f_*F$, where $f_*E_1,f_*E_2\in\cC^{NF}(Y)$ and $f_*F$ is a Frobenius finite bundle on $Y$. Hence $f_*E\cong f_*E_1/f_*E_2\in\cC^F(Y)$.

(3) For the case $*= Loc$. 

If $E\in\cC^{Loc}(X)$,  there exists $n$ such that $F_X^{n*}E\cong \cO_X^{\oplus r}$. We have 
$F_{Y}^{n*}f_*E \cong f_*F_{X}^{n*}E \cong f_*\cO_X^{\oplus r}\cong \cO_Y^{\oplus r}$.
Hence $f_*E\in\cC^{Loc}(Y)$.

(4) For the case $*= \acute{e}t$.
 
 If $E\in\cC^{\acute{e}t}(X)$, then there exists a finite \'etale covering $\phi:P\rightarrow X$ such that $\phi^*E\cong \cO_P^{\oplus r}$. Consider the following commutative diagram:
    \[
\begin{tikzcd}
P \arrow[r, "\phi"] \arrow[d, "f'"] \arrow[rd] & X \arrow[d, "f"] \\
Q \arrow[r, "\phi'"]                     & Y               
\end{tikzcd}
    \]
where $P \rightarrow Q \rightarrow Y$ is the Stein factorization of the composite morphism $P \rightarrow X \rightarrow Y$.
Then $\phi':Q\rightarrow Y$ is a finite \'etale covering, $f'_*\cO_P = \cO_Q$ and $f'^*\phi'^*f_*E\cong \phi^*f^*f_*E \cong  \phi^*E \cong \cO_{P}^{\oplus r}$.
Applying $f'_*$, by projection formula, we have 
    \[
    \phi'^*f_*E \cong \phi'^*f_*E \otimes_{\cO_Q}f'_*\cO_P \cong f'_*(f'^*\phi'^*f_*E \otimes_{\cO_P}\cO_P) \cong f'_*f'^*\phi'^*f_*E  \cong f'_*\cO_{P}^{\oplus r} \cong \cO_{Q}^{\oplus r}.
    \]
Hence $f_*E \in \cC^{\acute{e}t}(Y)$.

(5) For the case $* = uni$.

If $E\in\cC^{uni}(X)$, there exists a filtration
    \[
    0\hookrightarrow E_1\hookrightarrow \cdots\hookrightarrow E_n=E,
    \]
such that $E_{i+1}/E_i\cong\mathcal{O}_X$ for any $i$. Applying the functor $f_*$, we obtain a filtration of $f_*E$:
    \[
    0\hookrightarrow f_*E_1\hookrightarrow \cdots\hookrightarrow f_*E_n=f_*E,
    \]
such that $f_*E_{i+1}/f_*E_i\cong f_*(E_{i+1}/E_i)\cong f_*\cO_X =\cO_Y$ for any $i$. Hence $f_*E\in\cC^{uni}(Y)$.

(6) For the case $* \in \{EN, EF, ELoc, E\acute{e}t \}$.

Let $\bullet \in \{N, F, Loc,\acute{e}t\}$. For any $E \in \overline{\cC^{\bullet}(X)}$, consider the filtration
\[
    0\hookrightarrow E_1 \hookrightarrow \cdots\hookrightarrow E_n=E,
\]
such that $E^i=E_{i+1}/E_i\in\cC^{\bullet}(X)$ for any $i$. Applying $f_*$,
we obtain a filtration of $f_*E$:
    \[
    0\hookrightarrow f_*E_1\hookrightarrow \cdots\hookrightarrow f_*E_n=f_*E,
    \]
such that $f_*E_{i+1}/f_*E_i\cong f_*(E_{i+1}/E_i) =f_*(E^i) \in \cC^{\bullet}(Y)$ for any $i$. Hence $f_*E \in \overline{\cC^{\bullet}(Y)}$.
\end{proof}

\begin{Corollary}\label{numericalinv}
    Let $k$ be a perfect field, $f:X\rightarrow Y$ a birational morphism between normal varieties proper over $k$ with $Y$ smooth. Then for any $* \in \{S, N, EN, F, EF, Loc, ELoc, \acute{e}t, E\acute{e}t, uni\},$ the natural homomorphism $\pi^{*}(X,x)\rightarrow \pi^{*}(Y,y)$ is an isomorphism.
\end{Corollary}
\begin{proof}
    Let $\bar{k}$ be an algebraic closure of $k$. Consider the Cartesian diagram

\[
\begin{tikzcd}
X_{\bar{k}} 
  \arrow[r, "p_X"] 
  \arrow[d, "f_{\bar{k}}"'] 
  \arrow[dr, phantom, "\ulcorner", very near start]
& X 
  \arrow[d, "f"] \\
Y_{\bar{k}} 
  \arrow[r, "p_Y"] 
  \arrow[d]
  \arrow[dr, phantom, "\ulcorner", very near start]
& Y 
  \arrow[d] \\
\Spec \bar{k} 
  \arrow[r] 
& \Spec k .
\end{tikzcd}
\]
For any $E \in \cC^{NF}(X)$, we have $E \otimes_k \bar{k} \in \cC^{NF}(X_{\bar{k}})$. Since $k$ is perfect, we have $\bar{k}/k$ is a separable extension and $X_{\bar{k}}$ normal. The proof of \cite[Theorem~1.2]{HM11} also works for $f_{\bar{k}}: X_{\bar{k}} \rightarrow Y_{\bar{k}}$, hence we have $(f_{\bar{k}})_*(E \otimes_k \bar{k}) \in\cC^{NF}(Y_{\bar{k}})$. 
Since $f$ is proper and $p^*_Y$ is flat, by flat base change, we have
\[
(f_*E) \otimes_k \bar{k}= p_Y^*(f_*E) \cong (f_{\bar{k}})_*(p_X^*E)= (f_{\bar{k}})_* (E \otimes_k \bar{k}) \in \cC^{NF}(Y_{\bar{k}}).
\]
 By Lemma~\ref{basechange}, we have $f_*E \in \cC^{NF}(Y)$. 
Then  the natural homomorphism $\pi^{S}(X,x)\rightarrow \pi^{S}(Y,y)$ is an isomorphism by Corollary~\ref{birmor}. 

By Proposition~\ref{birinv:underS}, the natural homomorphism $\pi^{*}(X,x)\rightarrow \pi^{*}(Y,y)$ is an isomorphism for any $* \in \{N, EN, F, EF, Loc, ELoc, \acute{e}t, E\acute{e}t, uni\}$.
\end{proof}

\begin{Theorem}\label{birinvariance}
Let $k$ be a perfect field, $f \colon X \dashrightarrow Y$
be a birational map between smooth projective varieties over $k$, $x \in X(k)$ lying over $y \in Y(k)$, $* \in \{S, N, EN, F, EF, Loc, ELoc, \acute{e}t, E\acute{e}t, uni\}$. 
Then there exists a natural isomorphism
$\pi^{*}(X,x)\cong \pi^{*}(Y,y)$.
\end{Theorem}
\begin{proof}
Let $Z$ be the normalization of the closure of graph of $f$, $z \in Z(k)$ lying over $x$ and $y$. Consider the commutative diagram
\[
\begin{tikzcd}
& Z \arrow[ld, "p"'] \arrow[rd, "q"] & \\
X \arrow[rr, dashed, "f"'] & & Y.
\end{tikzcd}
\]  
 Let $* \in \{S, N, EN, F, EF, Loc, ELoc, \acute{e}t, E\acute{e}t, uni \}$.
 The natural homomorphisms $\pi^*(Z, z) \rightarrow \pi^*(X, x)$ and $\pi^*(Z, z) \rightarrow \pi^*(Y, y)$ are isomorphisms by Corollary~\ref{numericalinv}. Hence there exists a natural isomorphism
$\pi^{*}(X,x)\cong \pi^{*}(Y,y)$.
\end{proof}

\begin{Corollary}\label{cor:stratified}
Let $k$ be a perfect field of characteristic $p>0$, and let $f: X\rightarrow Y$
be a birational morphism between smooth projective  varieties over $k$, $x \in X(k)$ lying over $y \in Y(k)$.  Then the induced homomorphism
$\pi^{str}(X,x)\rightarrow \pi^{str}(Y,y)$
is also an isomorphism.
\end{Corollary}

\begin{proof}
By Corollary~\ref{numericalinv},  the induced homomorphism $\pi^{\acute{e}t}(X,x)\rightarrow \pi^{\acute{e}t}(Y,y)$ is an isomorphism. Then by \cite[Theorem II]{SZ25}, the induced homomorphism
$\pi^{str}(X, x)\rightarrow \pi^{str}(Y, y)$
is also an isomorphism.
\end{proof}

\section*{Acknowledgements}
The first author is partially supported by Applied Basic Research Programs of Science and Technology Commission Foundation of Shanghai Municipality(22JC1402700), National Natural Science Foundation of China(Grant No. 12171352).


\begin{thebibliography}{99}

\bibitem{AdAm25} P. Adroja and S. Amrutiya, \emph{On an extension of Nori and local fundamental group schemes}, Comm. Algebra {\bf 53} (2025), no.~10, 4241--4255.

\bibitem{AmBi10} S. Amrutiya, I. Biswas, \emph{On the F-fundamental group scheme}, Bull. Sci. Math. {\bf 134} (2010), no.~5, 461--474.

\bibitem{Am20}
S.~Amrutiya,
\emph{A note on certain Tannakian group schemes},
Archivum Mathematicum \textbf{56} (2020), no.~1, 21--29.

\bibitem{CR11}
A.~Chatzistamatiou and K.~R\"ulling,
\emph{ Higher direct images of the structure sheaf in positive characteristic}, Algebra Number Theory. \textbf{5} (2011), no.~6, 693--775.

\bibitem{CLM22} R. Cheng, C. Lian and T. Murayama, Projectivity of the moduli of curves, in {\it Stacks Project Expository Collection (SPEC)}, 1--43, London Math. Soc. Lecture Note Ser., 480, Cambridge Univ. Press, Cambridge (2022).

\bibitem{dosSantos07}
J.~P.~dos Santos,
\newblock {\em Fundamental group schemes for stratified sheaves},
\newblock J. Algebra \textbf{317} (2007), no.~2, 691--713.

\bibitem{Gro60} A. Grothendieck, \emph{Revêtements étales et groupe fondamental}, Séminaire de Géométrie Algébrique du Bois-Marie, (SGA 1), 1960/61, Lecture Notes in Mathematics, Vol.~224, Springer-Verlag, Berlin (1971).


\bibitem{HM11} A. Hogadi and V.~B. Mehta, Birational invariance of the $S$-fundamental group scheme, Pure Appl. Math. Q. {\bf 7} (2011), no.~4, Special Issue: In memory of Eckart Viehweg, 1361--1369.

\bibitem{Lan11} A. Langer, \emph{On the S-fundamental group scheme}, Ann. Inst. Fourier (Grenoble) {\bf 61} (2011), no.~5, 2077--2119 (2012).

\bibitem{LiTian26}
L. Li, N. Tian, \emph{The K\"unneth Formula of Fundamental Group Schemes},
\url{https://doi.org/10.48550/arXiv.2602.14207}.

\bibitem{MeSu08} V.~B. Mehta and S. Subramanian, \emph{Some remarks on the local fundamental group scheme}, Proc. Indian Acad. Sci. Math. Sci. {\bf 118} (2008), no.~2, 207--211.

\bibitem{Nor82} M. V. Nori, \emph{The fundamental group-scheme}, Proc. Indian Acad. Sci. Math. Sci. {\bf 91} (1982), no.~2, 73--122.

\bibitem{Ota17} S. Otabe, \emph{An extension of Nori fundamental group}, Comm. Algebra {\bf 45} (2017), no.~8, 3422--3448.

\bibitem{stacks} The {Stacks project authors}, \emph{The Stacks project}, \url{https://stacks.math.columbia.edu}, 2026.

\bibitem{SZ25}
X.~Sun and L.~Zhang,
{\em The \'etale fundamental group and $F$-divided sheaves in characteristic $p>0$},
\url{https://doi.org/10.48550/arXiv.2509.24360}.

\bibitem{Yeku19}
A.~Yekutieli,
\emph{ Derived Categories},
Cambridge Studies in Advanced Mathematics, Vol.~183,
Cambridge University Press (2019).

\end{thebibliography}
\end{document}